\documentclass[11pt]{article}
\usepackage{tikz}     

\usepackage{setspace}
\usepackage{amsthm,amsfonts,amssymb,epsfig,graphics,amsmath,amsbsy,subfigure}
\usepackage{amsthm,amsfonts,amssymb,epsfig,graphics,amsmath,amsbsy}
\usepackage{tikz}
\newcommand{\intL}{\int\limits}
\newcommand{\half}{^\infty_0 }
\newcommand{\intR}{\int\limits_{\mathbb{R}} }
\newcommand{\intRR}{\int\limits_{\mathbb{R}^2} }
\newcommand{\RR}{\mathbb{R}}
\newcommand{\xx}{\mathbf{x}}

\newtheorem{thm}{Theorem}
\newtheorem{defi}[thm]{Definition}
\newtheorem{rmk}[thm]{Remark}

\newtheorem{lem}[thm]{Lemma}

\title{A Radon-type transform arising in photoacoustic tomography with circular detectors}

\author{ Sunghwan Moon}
\date{Department of Mathematical Sciences\\[-0.1em]
\normalsize
Ulsan National Institute of Science and Technology\\[-0.1em]
\normalsize
Ulsan 689-798, Republic of Korea\\[-0.1em]
\normalsize
{\tt shmoon@unist.ac.kr}}

\begin{document}
\maketitle

\begin{abstract}
Photoacoustic tomography is the most well-known example of a hybrid imaging method. 
 In this article, we define a Radon-type transform arising in a version of photoacoustic tomography that uses integrating circular detectors and describe how the Radon transform integrating over all circles with a fixed radius is determined from this Radon-type transform.
Here we consider three situations: when the centers of the circular detectors are located on a cylinder, on a plane, and on a sphere.

This transform is similar to a toroidal Radon transform, which maps a given function to its integrals over a set of tori. 
We also study this object.
\end{abstract}


\section{Introduction}

 Photoacoustic Tomography (PAT) is a noninvasive medical imaging technique based on the reconstruction of an internal photoacoustic source. Its principle is based on the excitation of high bandwidth acoustic waves with pulsed non-ionizing electromagnetic energy \cite{xuw06,zangerls09,zangerlsh09}. Ultrasound imaging often has high resolution but displays low contrast. Optical or radio-frequency EM illumination, on the other hand, gives an enormous contrast between unhealthy and healthy tissues, although it has low resolution. The photoacoustic effect, which was discovered by A.G. Bell in 1880, is the underlying physical principle of PAT. 
PAT can provide information about the chemical composition as well as the optical structure of an object. 
 
In PAT, one induces an acoustic wave inside an object of interest by delivering optical energy~\cite{kuchmentk08,xuw06}, and then one measures the acoustic wave to a surface outside of the object of interest~\cite{burgholzerbmghp07,xuw06,zangerls09}.
The initial data of the three dimensional wave equation contain diagnostic information. 
{\color{black}One of the mathematical problems of PAT} boils down to recovering this initial pressure field. 

The type of detector most often studied is a point transducer, which approximately measures the pressure at a given point. 
{\color{black}However, it is} difficult to manufacture small detectors with high bandwidth and sensitivity. 
Hence, various other types of detectors to measure the acoustic data have been introduced, such as line detectors, planar detectors, cylindrical detectors and circular detectors. 
Measurements are modeled by the integrals of pressure over the shape of the detectors. 

{\color{black}Works~\cite{zangerls09,zangerls10,zangerlsh09} have} dealt with PAT with the circular detectors. 
 {They showed that the data from PAT with circular detectors is the solution of a certain initial value problem, and they converted the problem of recovering the initial pressure field into an inversion problem for the circular Radon transform using this fact.}
Also, they assume that the circular detectors are centered on a cylinder in~\cite{zangerls09,zangerlsh09} or the circular detectors are of different radii and are lying on a surface of a sphere in~\cite{zangerls10}.
In our approach, we define a new Radon-type transform arising in this version of PAT, and consider the situation when the set of the  {centers of the circular detectors} is a cylinder (only this situation is discussed in previous works~\cite{zangerls09,zangerlsh09}) and two more situations: when the set of the  {centers of the circular detectors} is a plane or a sphere (this spherical geometry is different from that in~\cite{zangerls10}). 
This transform is similar to a toroidal Radon transform, which maps a given function into  {the set of its integrals over tori with respect to a certain non-standard measure}; we also study this mathematically similar object. 
  
This paper is organized as follows.
Section~\ref{defiandwork} is devoted to a Radon-type transform arising in PAT with circular detectors. We reduce this Radon-type transform to the Radon transform on circles with a fixed radius.
In section~\ref{sec:torus}, we define a toroidal Radon transform and reduce this transform to the circular Radon transform. 
\section{PAT with circular integrating detectors}\label{defiandwork}
In PAT, the acoustic pressure $p(\xx,t)$ satisfies the following initial value problem:
\begin{equation}\label{eq:pdeofpat}
\begin{array}{cc}
 \partial^2_tp(\xx,t)=\triangle_\xx p(\xx,t)\qquad&(\xx,t)=(x_1,x_2,x_3,t)\in\RR^3\times(0,\infty),\\
p(\xx,0)=f(\xx)&\xx\in\RR^3,\\
\partial_tp(\xx,0)=0 &\xx\in\RR^3.
\end{array}
\end{equation}
(We assume that the sound speed is equal to one everywhere including the interior of object.) The goal of PAT is to recover the initial pressure $f$ from measurements of $p$ outside the support of $f$.

\begin{figure}
\begin{center}
  \begin{tikzpicture}[>=stealth,scale=0.7]
    \draw[->] (-11,2.4) -- (-11,3);
    \draw[densely dotted] (-11,-1) -- (-11,2.4);
    \draw[->] (-9,-1) -- (-8,-1);
    \draw[densely dotted] (-11,-1) -- (-9,-1);
    \draw[->] (-12,-2) -- (-12.5,-2.5);
    \draw[densely dotted] (-11,-1) -- (-12,-2);
    \draw[very thick] (-11,1.9) ellipse (2 and 0.5) ;
    \draw[very thick, dashed] (-9,-1.5) arc (0:180:2 and 0.5) ;
    \draw[very thick] (-13,-1.5) arc (180:360:2 and 0.5) ;
    \draw[very thick] (-13,-1.5) -- (-13,1.9);
    \draw[very thick] (-9,-1.5) -- (-9,1.9);
    \draw[dashed] (-9.6,0.2) -- (-9.6,2.4);
    \draw[<->,loosely dashed] (-11,-1.5) -- (-9,-1.5);    
    \draw[blue] (-11,2.4) ellipse (0.5 and 0.2) ;
    \draw[blue] (-10.2,2.35) ellipse (0.5 and 0.2) ;
    \draw[blue] (-9.6,2.26) ellipse (0.5 and 0.2) ;
    \draw[blue] (-9.6,1.96) ellipse (0.5 and 0.2) ;
    \draw[blue] (-9.6,1.66) ellipse (0.5 and 0.2) ;
    \draw[blue] (-4.8,2.1) ellipse (0.5 and 0.2) ;
    \draw[blue] (-4.8,1.8) ellipse (0.5 and 0.2) ;
    \draw[blue] (-4.8,1.5) ellipse (0.5 and 0.2) ;
    \draw[blue] (-6,2.1) ellipse (0.5 and 0.2) ;
    \draw[blue] (-6,1.8) ellipse (0.5 and 0.2) ;
    \draw[blue] (-6,1.5) ellipse (0.5 and 0.2) ;
    \draw[dashed] (-4.8,0.2) -- (-4.8,2.5);
    \draw[->] (-6,-1) -- (-6,3);
    \draw[->] (-6,-1) -- (-3,-1);
    \draw[->] (-6,-1) -- (-7,-2);
    \draw[->] (0,2) -- (0,3);
    \draw[densely dotted] (0,0) -- (0,2);
    \draw[->] (2,0) -- (3,0);
    \draw[densely dotted] (0,0) -- (2,0);
    \draw[->] (-1.4,-1.4) -- (-2,-2);
    \draw[densely dotted] (0,0) -- (-1.5,-1.5);
    \draw[<->,loosely dashed] (0,0) -- (-1,-.4); 
    \draw[very thick] (0,0) circle (2) ;
    \draw[dashed] (2,0) arc (0:180:2 and 0.5) ;
    \draw (-2,0) arc (180:360:2 and 0.5) ;
    \draw[blue,<->] (1.6,1.2) -- (2.1,1.2);
    \draw[blue,<->] (-4.8,2.1) -- (-4.3,2.1);
    \draw[blue,<->] (-9.6,2.26) -- (-9.1,2.26);
    \draw[blue] (1,1.73205) ellipse (0.5 and 0.2) ;
    \draw[blue] (1.3,1.51987) ellipse (0.5 and 0.2) ;
        \draw[blue] (1.6,1.2) ellipse (0.5 and 0.2) ;
    \node at (-10,-1.7) {$R$};
    \node at (-0.5,0) {$R$};
    \node at (2.2,1.5) {$r_{det}$};
    \node at (-4.4,2.5) {$r_{det}$};
    \node at (-9.4,2.6) {$r_{det}$};
    \node at (-3,-0.5) {$x_2$};
    \node at (-6.3,3.2) {$x_3$};
    \node at (-6.8,-2.2) {$x_1$};
    \node at (3,-0.5) {$x_2$};
    \node at (0.5,3.2) {$x_3$};
    \node at (-2.2,-2.2) {$x_1$};
        \node at (-8,-0.5) {$x_2$};
    \node at (-11,3.2) {$x_3$};
    \node at (-11.7,-2.3) {$x_1$};
     \node at (-11,-2.8) {(a)};
    \node at (-6,-2.8) {(b)};
    \node at (0,-2.8) {(c)};
  \end{tikzpicture}
\caption{The  {centers of the circular detectors} on (a) a cylinder, (b) a plane, and (c) a sphere}
\end{center}\label{fig:circleTAT}
\end{figure}
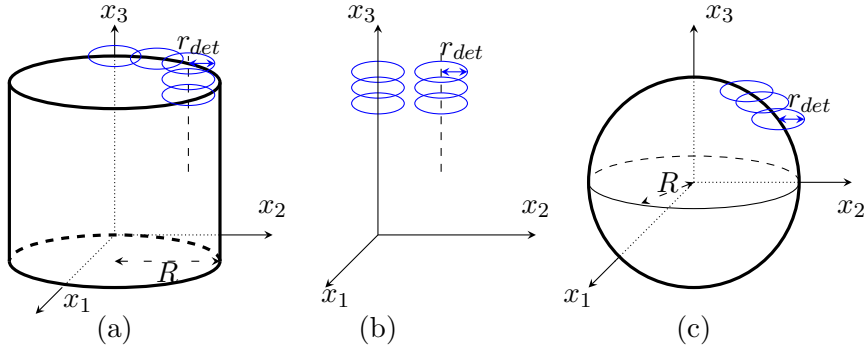 
Throughout this section, it is assumed that the initial pressure field $f$ is smooth and circular detectors are parallel to the $x_1x_2$-plane. 
As mentioned before, three geometries will be studied. 
Let the  {centers of the circular detectors}  be located on a subset $A$ of $\RR^3$. 
We consider three cases: $A$ is the cylinder $\partial B_R^2(0)\times\RR$, the $x_2x_3$-plane, or the sphere $\partial B^3_R(0)$ (see figure 1).
Here $B_R^k(\xx)=B^k(\xx,R)$ is a ball in $\RR^k$ centered at $\xx\in\RR^k$ with radius $R$. 
In other words, it is assumed that the acoustic signals are measured by a stack of parallel circular detectors where these circles are centered on a cylinder $\partial B_R^2(0)\times\RR$, on the $x_2x_3$-plane, or on the sphere $\partial B^3_R(0)$ and their radii are a constant $r_{det}$. 

The measured data $P(\mathbf a,t)$ for $(\mathbf a,t)\in A\times(0,\infty)$ can be written as
$$
P(\mathbf a,t)=\frac{1}{2\pi}\intL^{2\pi}_0p(\mathbf  a+(r_{det}\vec\alpha,0),t)d\alpha,
$$ 
where $\vec\alpha=(\cos\alpha,\sin\alpha)\in S^1$.
Also, it is a well-known fact that
$$
p(\xx,t)=\partial_t\left(\frac{1}{4\pi t}\intL_{\partial B^3(\xx,t)}f(\vec\beta)dS(\vec\beta)\right),
$$
is a solution of the IVP~\eqref{eq:pdeofpat}.
Here $dS$ is the measure on the sphere and 
$$
\vec\beta=(\cos\beta_1\sin\beta_2,\sin\beta_1\sin\beta_2,\cos\beta_2)\in S^2.
$$
Hence $P(\mathbf a,t)$ becomes
$$
\begin{array}{ll}
P(\mathbf a,t)&\displaystyle=\frac{1}{2\pi}\intL^{2\pi}_0\partial_t\left(\frac{1}{4\pi t}\intL_{\partial B^3(\mathbf a+(r_{det}\vec\alpha,0),t)}f(\vec\beta)dS(\vec\beta) \right)d\alpha\\
&\displaystyle=\frac{1}{8\pi^2}\partial_t\left(t\intL^{2\pi}_0\intL^{\pi}_0\intL^{2\pi}_0
f(\mathbf a+(r_{det}\vec\alpha,0)+t\vec\beta)\sin\beta_2 d\beta_1 d\beta_2 d\alpha\right). 
\end{array}
$$
Let us define a transform $\mathcal R_P$ by
$$
\begin{array}{l}
\displaystyle \mathcal{R}_Pf(\mathbf a,t):=\intL^{2\pi}_0\intL^{2\pi}_0\intL^{\pi}_0 f(\mathbf a+(r_{det}\vec\alpha,0)+t\vec\beta)\sin\beta_2  d\beta_2 d\beta_1d\alpha.
\end{array}
$$
We will demonstrate a relation between the Radon transform on circles with a fixed radius$-$a well studied problem$-$and $\mathcal{R}_Pf$.
This will allow us to recover $f$ from $P$.

\subsection{Reconstruction}\label{recon}

Let a transform $M_{r_{det}}f$ be defined by
$$
M_{r_{det}}f(\xx):=\intL^{2\pi}_0f(r_{det}\vec\alpha+(x_1,x_2),x_3)d\alpha,
$$
where $\xx=(x_1,x_2,x_3)\in\RR^3$.

We will show that $M_{r_{det}}f$ can be obtained from $\mathcal{R}_Pf$ when $A$ is a cylinder, a plane or a sphere.

We can easily find the inversion of the Radon transform $M_{r_{det}}f$ over all circles with a fixed radius as follows:
Taking the 2-dimensional Fourier transform of $M_{r_{det}}f$ with respect to $(x_1,x_2)$, we have for $\boldsymbol\xi=(\xi_1,\xi_2)\in\RR^2,$
$$
\widehat{M_{r_{det}}f}(\boldsymbol\xi,x_3)=2\pi\hat{f}(\boldsymbol\xi,x_3)J_0(r_{det}|\boldsymbol\xi|),
$$
where  $J_0(s)$ is the Bessel function of order zero, and 
$\widehat {M_{r_{det}}f}$ and $\hat f$ are the 2-dimensional Fourier transforms of $M_{r_{det}}f$ and $f$ with respect to $(x_1,x_2)$.
Hence we can reconstruct $f$ though
$$
f(\xx)=\frac1{4\pi^2}\intRR\widehat{M_{r_{det}}f}(\boldsymbol\xi,x_3)/J_0(r_{det}|\boldsymbol\xi|) e^{i\boldsymbol\xi\cdot(x_1,x_2)}d\boldsymbol\xi.
$$
\begin{rmk}\label{rmk:pompeiu}
When we have two circular detectors with different radii, say $r_1,r_2$, we have two different values $M_{r_{1}}f,M_{r_{2}}f$ for each $\xx$, i.e., two Radon transforms on circles with different fixed radii.
Some works~\cite{berensteingy90,thangavelu94,zalcman80} show how $f$ can be reconstructed from $M_{r_{1}}f,M_{r_{2}}f$ under a certain assumption. 
\end{rmk}

\subsubsection{Cylindrical geometry}
Let $f$ be a smooth function supported in the solid cylinder $B^2_R(0)\times\RR=\{\xx=(x_1,x_2,x_3)\in\RR^3:x_1^2+x_2^2\leq R\}$. Let the centers of the circular detectors be located on the cylinder $A=\partial B^2(0)\times\RR$.
We can represent $\mathbf a\in A$ by $(R\vec\theta,z)\in\partial B^2_R(0)\times\RR$ for $(\vec\theta,z)\in S^1\times\RR$.
Then $\mathcal R_Pf(R\vec\theta,z,t)$ is equal to
$$
\intL^{2\pi}_0\intL^{2\pi}_0\intL^{\pi}_0
f((R\vec\theta+r_{det}\vec\alpha,z)+t\vec\beta)\sin\beta_2  d\beta_2 d\beta_1d\alpha.
$$
Consider the definition of $\mathcal R_Pf$. 
The inner integral with respect to $\beta_2$ in the definition formula of $\mathcal R_Pf$ can be thought of as the circular Radon transform with weight $\sin\beta_2$.
We will first remove this integral by applying the technique previously used to derive an inversion formula for the circular Radon transform.

Let us define the operator 
 $\mathcal{R}^\#_P$ for an integrable function $g$ on $\partial B_R^2(0)\times\RR\times[0,\infty)$ by
$$
\mathcal{R}^\#_Pg(R\vec\theta,x_3,\rho)=\displaystyle\intR g(R\vec\theta,z,\sqrt{(z-x_3)^2+\rho^2})|(z-x_3,\rho)|dz.
$$
\begin{lem}\label{lem:relation}
 Let $f\in C^\infty_c(B^2_R(0)\times\RR)$.
Then we have
\begin{equation}\label{eq:circularcylinder}
\intL^{2\pi}_0\intL^{2\pi}_0f(R\vec\theta+r_{det}\vec\alpha+t\vec\beta_1,x_3) d\beta_1  d\alpha=-\frac{1}{\pi^2 t}\mathcal H_t\partial_t\mathcal{R}^\#_P\mathcal{R}_Pf(R\vec\theta,x_3,t). 
\end{equation}
Here $\mathcal H_th$ is the Hilbert transform of $h$ with respect to $t$ defined by
$$
\mathcal H_th(t)=\frac1\pi P.V.\intR h(\tau)\frac{d\tau}{t-\tau},
$$
where $P.V.$ means the principal value.
\end{lem}
To prove this theorem, we follow a similar method to the one used in~\cite{andersson88,fawcett85,nattererw01,nilsson97}.
\begin{proof}
By definition, $\mathcal{R}_Pf(R\vec\theta,z,t)$ can be written as
$$
-\intL^{2\pi}_0\intL^{2\pi}_0\int\limits^1_{-1}f(R\vec\theta+r_{det}\vec\alpha+t\sqrt{1-s^2}\vec\beta_1, z+ts)dsd\beta_1 d\alpha.
$$
Taking the Fourier transform of $\mathcal{R}_Pf$ with respect to $z$ yields 
$$
\widehat{\mathcal{R}_Pf}(R\vec\theta,\xi_1,t)=-\intL^{2\pi}_0\intL^{2\pi}_0\int\limits^1_{-1}\hat{f}(R\vec\theta+r_{det}\vec\alpha+t\sqrt{1-s^2}\vec\beta_1, \xi_1)e^{its\xi_1}dsd\beta_1 d\alpha,
$$
where $\hat{f}$ and $\widehat{\mathcal{R}_Pf}$ are the 1-dimensional Fourier transforms of $f$ and $\mathcal{R}_Pf$ with respect to $x_3$ and $z$, respectively.
Taking the Hankel transform of order zero of $t\widehat{\mathcal{R}_Pf}$ with respect to $t$, we have 
$$
\begin{array}{ll}
H_0(t\widehat{\mathcal{R}_Pf})(R\vec\theta,\xi_1,\eta)\\
=-\displaystyle\int\limits\half\intL^{2\pi}_0\intL^{2\pi}_0\int\limits^1_{-1}\hat{f}(R\vec\theta+r_{det}\vec\alpha+t\sqrt{1-s^2}\vec\beta_1, \xi_1)e^{its\xi_1}dsd\beta_1 d\alpha \;t^2J_0(t\eta)dt\\
=-2\displaystyle\int\limits\half\intL^{2\pi}_0\intL^{2\pi}_0\int\limits^1_0\hat{f}(R\vec\theta+r_{det}\vec\alpha+t\sqrt{1-s^2}\vec\beta_1, \xi_1)t^2J_0(t\eta)\cos(ts\xi_1)dsd\beta_1 d\alpha dt\\
=-2\displaystyle\intL^{2\pi}_0\intL^{2\pi}_0\int\limits\half\int\limits\half\hat{f}(R\vec\theta+r_{det}\vec\alpha+b\vec\beta_1, \xi_1)b\cos(\rho\xi_1) J_0(\eta\sqrt{\rho^2+b^2})d\rho dbd\beta_1d\alpha,
  \end{array}
  $$
  where in the last line, we made a change of variables $(t,s)\rightarrow (\rho,b)$ where $t=\sqrt{\rho^2+b^2}$ and $s=\rho/\sqrt{\rho^2+b^2}$.
We use the following identity: for $0<\xi_1<a$,
\begin{equation}\label{eq:batemann}
\displaystyle\int\limits\half J_0(a\sqrt{\rho^2+b^2})\cos(\rho\xi_1)d\rho=\left\{\begin{array}{ll}\dfrac{\cos(b\sqrt{a^2-\xi_1^2})}{\sqrt{a^2-\xi_1^2}} &\mbox{ if } 0<\xi_1<a,\\
0&\mbox{ otherwise}\end{array}\right.
\end{equation}
(see~\cite[p.55 (35) vol.1]{batemann}). Using this identity, $H_0(t\widehat{\mathcal{R}_Pf})(R\vec\theta,\xi_1,\eta)$ is equal to
\begin{equation*}
-2\left\{\begin{array}{ll}\displaystyle\intL^{2\pi}_0\intL^{2\pi}_0\int\limits\half\hat{f}(R\vec\theta+r_{det}\vec\alpha+b\vec\beta_1, \xi_1)b\dfrac{\cos(b\sqrt{\eta^2-\xi_1^2})}{\sqrt{\eta^2-\xi_1^2}}dbd\beta_1d\alpha\;&\mbox{ if } 0<\xi_1<\eta,\\
0&\mbox{ otherwise.}\end{array}\right.
 \end{equation*}
Substituting $\eta=\sqrt{\xi_1^2+\xi_2^2}$ yields
\begin{equation}\label{eq:hankel}
H_0(t\widehat{\mathcal{R}_Pf})(R\vec\theta,\xi_1,|\boldsymbol\xi|)=-2\displaystyle\intL^{2\pi}_0\intL^{2\pi}_0\int\limits\half\hat{f}(R\vec\theta+r_{det}\vec\alpha+b\vec\beta_1, \xi_1)\dfrac{b}{\xi_2}\cos(b\xi_2)dbd\beta_1d\alpha.
 \end{equation}
The inner integral in the right hand side of the last equation is the Fourier cosine transform with respect to $b$, so taking the Fourier cosine transform of~\eqref{eq:hankel}, we get
\begin{equation}\label{relationhankelandfourier}
\begin{array}{l}
\displaystyle\intL^{2\pi}_0\intL^{2\pi}_0 \hat{f}(R\vec\theta+r_{det}\vec\alpha+s\vec\beta_1, \xi_1)sd\beta_1d\alpha\\
\displaystyle= -\pi^{-1}\int\limits\half H_0(t\widehat{\mathcal{R}_Pf})(R\vec\theta,\xi_1,|\boldsymbol\xi|)\cos(s\xi_2)\xi_2 d\xi_2,
\end{array}
\end{equation}
where $\hat{f}$ is the Fourier transform of $f$ with respect to the last variable $x_3$.

We change the right hand side of~\eqref{relationhankelandfourier} into a term containing the operator $\mathcal{R}_P ^\#$.
Taking the Fourier transform of $\mathcal{R}^\#_Pg$ on $\partial B^2(0)\times\RR^2$ with respect to $(z,\rho)$ yields 
\begin{equation}\label{eq:relationhankelandback}
\begin{array}{ll}
   \widehat{\mathcal{R}^\#_Pg}(R\vec\theta, \boldsymbol\xi)=\displaystyle\intR\intR e^{-i(x_3,\rho)\cdot \boldsymbol\xi} \mathcal{R}^\#_Pg(R\vec\theta,x_3,\rho)dx_3d\rho\\
   =\displaystyle\intR\intR e^{-i(x_3,\rho)\cdot \boldsymbol\xi} \intR  |(x_3-z,\rho)|g(\theta,z,\sqrt{(z-x_3)^2+\rho^2} )dzdx_3d\rho\\
=\displaystyle\intR e^{-i\xi_1\cdot z} \intR\intR e^{-i(x_3-z,\rho)\cdot\boldsymbol\xi} |(x_3-z,\rho)| g(\theta,z,\sqrt{(z-x_3)^2+\rho^2} )dx_3d\rho dz\\
=\displaystyle\intR e^{-i\xi_1\cdot z}\intR\intR e^{-i(x_3,\rho)\cdot \boldsymbol\xi} |(x_3,\rho)|g(\theta,z,|(x_3,\rho)| )dx_3d\rho dz\\
=2\pi\displaystyle\intR e^{-i\xi_1\cdot z}H_0(tg)(R\vec\theta,z,|\boldsymbol\xi| )dz\\
=2\pi H_0(t\hat{g})(R\vec\theta,\xi_1,|\boldsymbol\xi| ),
  \end{array}
  \end{equation}
  where $\widehat{\mathcal{R}^\#_Pg}$ is the Fourier transform with respect to the last variable $(x_3,\rho)$.
Combining~\eqref{eq:relationhankelandback} with~\eqref{relationhankelandfourier}, we have for $g=\mathcal{R}_Pf$,
\begin{equation*}\label{relationfourierandback}
\begin{array}{ll}
\displaystyle\intL^{2\pi}_0\intL^{2\pi}_0\hat{f}(R\vec\theta+r_{det}\vec\alpha+s\vec\beta_1, \xi_1)d\beta_1d\alpha &\displaystyle
=-\frac{1}{2\pi^2 s} \int\limits\half \widehat{\mathcal{R}^\#_Pg}(R\vec\theta, \boldsymbol\xi) \cos(s\xi_2)\xi_2 d\xi_2\\
&\displaystyle=-\frac{1}{\pi^2 s} \intR \widehat{\mathcal{R}^\#_Pg}(R\vec\theta, \boldsymbol\xi) e^{is\xi_2}|\xi_2| d\xi_2\\
&\displaystyle=\frac{1}{\pi^2 s} \intR \widehat{\partial_t\mathcal{R}^\#_Pg}(R\vec\theta, \boldsymbol\xi) e^{is\xi_2}(i\operatorname{sgn}(\xi_2)) d\xi_2.
\end{array}
\end{equation*}
The fact that $\widehat{\mathcal H_th}(\xi)=(-i\operatorname{sgn}(\xi))\hat h(\xi)$ completes our proof.
\end{proof}

Again, the inner integral with respect to $\beta_1$ in the left hand side of~\eqref{eq:circularcylinder} is the circular Radon transform with centers on $\partial B^2_R(0)$ and radius $t$.
Hence, by applying an inversion formula of the circular Radon transform, we get $M_{r_{det}}f(\xx)$.

\begin{thm}\label{thm:circle}
Let $f$ be a smooth function supported in $B^2_R(0)\times\RR$.
Then for any $\xx\in\RR^3$, we have 
$$
M_{r_{det}}f(\xx)=\displaystyle\frac{1}{\pi R}\triangle_{x_1,x_2}\intL^{2\pi}_0 \mathcal{R}^\#_P\mathcal{R}_Pf(R\vec\theta,x_3, |(x_1,x_2)-R\vec\theta|)d\theta.
$$
\end{thm}
To prove this theorem, we follow the method discussed in~\cite{finchhr07}.
\begin{proof}
It is computed in~\cite{finchhr07} that
$$
\begin{array}{l}
\displaystyle\intL^{2\pi}_0\log \left||(x_1,x_2)-R\vec\theta|^2-|(y_1,y_2)-R\vec\theta|^2\right|d\theta\\
=2\pi R\log|(x_1,x_2)-(y_1,y_2)|+2\pi R\log R.
\end{array}
$$
For any measurable function $q$ on $\RR$, it is easily shown that
$$
\begin{array}{l}
\displaystyle\intL^{2(R+r_{det})}_0t\intL^{2\pi}_0\intL^{2\pi}_0
f(R\vec\theta+r_{det}\vec\alpha+t\vec\beta_1,x_3)d\beta_1d\alpha \;q(t)dt\\
\displaystyle=\intL^{2\pi}_0\intRR  f(R\vec\theta+r_{det}\vec\alpha+w,x_3)q(|\mathbf w|)d\mathbf wd\alpha.
\end{array}
$$
Applying this with $q(t)=\log\left|t^2-|(x_1,x_2)-R\vec\theta|^2\right|$ and making the change of variables $\mathbf y=(y_1,y_2)=R\vec\theta+t\vec\beta_1\in\RR^2$ give
$$
\begin{array}{l}
\displaystyle\intL^{2\pi}_0\intL^{2(R+r_{det})}_0\intL^{2\pi}_0\intL^{2\pi}_0
tf(R\vec\theta+r_{det}\vec\alpha+t\vec\beta_1,x_3)\log\left|t^2-|(x_1,x_2)-R\vec\theta|^2\right|d\beta_1d\alpha dt d\theta\\
=\displaystyle\intL^{2\pi}_0\intL^{2\pi}_0\intRR f(r_{det}\vec\alpha+\mathbf y,x_3)\log \left||(x_1,x_2)-R\vec\theta|^2-|\mathbf y-R\vec\theta|^2\right| d\mathbf y d\alpha d\theta\\
=\displaystyle 2\pi R\intL^{2\pi}_0\intRR f(r_{det}\vec\alpha+\mathbf y,x_3) (\log|(x_1,x_2)-\mathbf y|+\log R)d\mathbf y d\alpha,
\end{array}
$$
where in the last line, we used the Fubini-Tonelli theorem.
Since $(2\pi)^{-1}\log|(x_1,x_2)-(y_1,y_2)|+\log R$ is a fundamental solution of the Laplacian in $\RR^2$, we have
$$
\begin{array}{l}

\displaystyle\triangle\intL^{2\pi}_0\intL^{2(R+r_{det})}_0\intL^{2\pi}_0\intL^{2\pi}_0t
f(R\vec\theta+r_{det}\vec\alpha+t\vec\beta_1,x_3)\log\left|t^2-|(x_1,x_2)-R\vec\theta|^2\right|d\beta_1d\alpha dt d\theta\\
=R\displaystyle\intL^{2\pi}_0f((x_1,x_2)+r_{det}\vec\alpha,x_3)d\alpha,
\end{array}
$$
where $\triangle$ is the Laplacian on $(x_1,x_2)$.
Applying Lemma~\ref{lem:relation}, $M_{r_{det}}f(\xx)$ is equal to 
 $$
-\frac{1}{\pi^2 R}\triangle_{x_1,x_2}\intL^{2\pi}_0\intL^{2(R+r_{det})}_0\mathcal H_t\partial_t\mathcal{R}^\#_P\mathcal{R}_Pf(R\vec\theta,x_3, t) \log\left|t^2-|(x_1,x_2)-R\vec\theta|^2\right|dt d\theta.
$$
We note that $\mathcal H_t\partial_th=\partial_t\mathcal H_t h$,
$$ \log\left|t^2-|(x_1,x_2)-R\vec\theta|^2\right|= \log\left|t-|(x_1,x_2)-R\vec\theta|\right|+ \log\left|t+|(x_1,x_2)-R\vec\theta|\right|,
$$
and $\log|t|$ is $P.V.\frac1t$.
By integration by parts, we have 
$$
\begin{array}{ll}
M_{r_{det}}f(\xx)=&\displaystyle\frac{1}{\pi^2 R}\triangle_{x_1,x_2}\intL^{2\pi}_0P.V.\intL^{2(R+r_{det})}_0\frac{\mathcal H_t\mathcal{R}^\#_P\mathcal{R}_Pf(R\vec\theta,x_3, t)}{t-|(x_1,x_2)-R\vec\theta|}dt d\theta\\
&+\displaystyle\frac{1}{\pi^2 R}\triangle_{x_1,x_2}\intL^{2\pi}_0\intL^{2(R+r_{det})}_0\frac{\mathcal H_t\mathcal{R}^\#_P\mathcal{R}_Pf(R\vec\theta,x_3, t)}{t+|(x_1,x_2)-R\vec\theta|}dt d\theta.
\end{array}
$$ 
Since $\mathcal{R}^\#_P\mathcal{R}_Pf$ is even in $t$, so is $\mathcal H_t\mathcal{R}^\#_P\mathcal{R}_Pf$.
Substituting $t=-t$ in the second term gives
$$
M_{r_{det}}f(\xx)=\displaystyle\frac{1}{\pi^2 R}\triangle_{x_1,x_2}\intL^{2\pi}_0P.V.\intL^{2(R+r_{det})}_{-2(R+r_{det})}\frac{\mathcal H_t\mathcal{R}^\#_P\mathcal{R}_Pf(R\vec\theta,x_3, t)}{t-|(x_1,x_2)-R\vec\theta|}dt d\theta.
$$
The fact that $\mathcal H_t\mathcal H_th(t)=-h(t)$ completes our proof.
\end{proof}

\begin{rmk}\label{rmk:line}
When $A$ is a cylinder, we can reconstruct $f$ from $\mathcal{R}_Pf$ by applying Theorem~\ref{thm:circle} and the argument below Remark~\ref{rmk:pompeiu}.
\end{rmk}


\subsubsection{Planar geometry}
Let the  {centers of the circular detectors} be located on the $x_2x_3$-plane.
Then we can denote $\mathbf a\in A$ by $(y,z)\in \RR^2$.
Also, $\mathcal{R}_Pf$ is equal to zero if $f$ is an odd function in $x_1$.
We thus assume the function is even in $x_1$.
Then $\mathcal{R}_Pf$ can be written by
$$
\mathcal{R}_Pf(y,z,t)=\intL^{2\pi}_0\intL_{S^2}
f((0,y,z)+r_{det}(\vec\alpha,0)+t\vec\beta) dS(\vec\beta) d\alpha.
$$
Let $M_P$ be the the spherical Radon transform mapping a locally integrable function $f$ on $\RR^3$ into its integral over the set of spheres centered on the $x_2x_3$-plane:
$$
M_Pf(y,z,t)=\intL_{S^2}f((0,y,z)+t\vec\beta)dS(\vec\beta).
$$
Then we have
$$
\mathcal{R}_Pf(y,z,t)=\intL_{S^2}M_{r_{det}}f((0,y,z)+t\vec\beta)dS(\vec\beta)=M_P(M_{r_{det}}f)(y,z,t).
$$

It is well-known (see, e.g.~\cite{nattererw01,nilsson97}) that for $\boldsymbol\xi=(\xi_1,\xi_2,\xi_3)\in\RR^3$,
$$
\hat f(\boldsymbol\xi)=\frac{|\boldsymbol\xi||\xi_1|}{4\pi^3}\mathcal F (M_P^*M_Pf)(\boldsymbol\xi),
$$
where $\mathcal F f=\hat f$ is the $3$-dimensional Fourier transform of $f$, and 
for an integrable function $g$ on $\RR^2\times[0,\infty)$,
$$
M_P^*g(\xx)=\intRR g(y,z,\sqrt{x_1^2+(y-x_2)^2+(z-x_3)^2})dydz.
$$

\begin{thm}
Let $f\in C^\infty_c(\RR^3)$ be even in $x_1$.
Then we have
$$
M_{r_{det}}f(\xx)=\frac{1}{2^5\pi^6}\intL_{\RR^3}|\boldsymbol\xi||\xi_1|\mathcal F (M_P^*\mathcal R_Pf)(\boldsymbol\xi)e^{i\boldsymbol\xi\cdot\xx}d\boldsymbol\xi.
$$
\end{thm}
Now $f$ can be determined applying the argument below Remark~\ref{rmk:pompeiu}.
\begin{rmk}
Redding and Newsam derived another inversion formula for the spherical Radon transform $M_P$ in~\cite{reddingn01}. Using this inversion formula, we can also reconstruct $M_{r_{det}}f$ from $\mathcal R_Pf$.
\end{rmk}
\subsubsection{Spherical geometry}
Let the  {centers of the circular detectors} be located on the sphere $\partial B^3_R(0)$.
Then we can denote $\mathbf a\in A$ by $R\vec\omega\in \partial B^3_R(0)$ for $\vec\omega\in S^2$.
Then $\mathcal{R}_Pf$ can be written by
$$
\mathcal{R}_Pf(R\vec\omega,t)=\intL^{2\pi}_0\intL_{S^2}
f(R\vec\omega+r_{det}(\vec\alpha,0)+t\vec\beta) dS(\vec\beta) d\alpha.
$$
Let $M_S$ be the spherical Radon transform mapping a locally integrable function $f$ on $\RR^3$ into its integral over the spheres centered at the $\partial B_R^3(0)$:
$$
M_Sf(\vec\omega,t)=\intL_{S^2}f(R\vec\omega+t\vec\beta)dS(\vec\beta).
$$
Then we have
$$
\mathcal{R}_Pf(R\vec\omega,t)=\intL_{S^2}M_{r_{det}}f(R\vec\omega+t\vec\beta)dS(\vec\beta)=M_S(M_{r_{det}}f)(\vec\omega,t).
$$

It is well-known (see, e.g.~\cite{finchpr04}) that
$$
f(\xx)=-\frac{R }{2^3\pi^2 }\intL_{S^2}\frac{\left.\partial^2_tt^2M_Sf( \vec\omega,t)\right|_{t=|R\vec\omega-\xx|}}{|R\vec\omega-\xx|}dS(\vec\omega).
$$

\begin{thm}
Let $f\in C^\infty_c(B^3_R(0))$.
Then we have
$$
M_{r_{det}}f(\xx)=-\frac{R }{2^3\pi^2}\intL_{S^2}\frac{\left.\partial^2_tt^2\mathcal R_Pf(R\vec\omega,t)\right|_{t=|R\vec\omega-\xx|}}{|R\vec\omega-\xx|}dS(\vec\omega).
$$
\end{thm}
Again $f$ can be determined applying the argument below Remark~\ref{rmk:pompeiu}.
\begin{rmk}
Kunyansky derived two other inversion formulas for the spherical Radon transform $M_S$ in~\cite{kunyansky07,kunyansky071}. Using these inversion formulas, we can reconstruct $M_{r_{det}}f$ from $\mathcal R_Pf$.
\end{rmk}


\section{A toroidal Radon transform}\label{sec:torus}
As mentioned before, we study the toroidal Radon transform, which is a mathematically similar object to $\mathcal R_P$, in this section. 
(When integrating over the tori, the standard area measure is not used.)
Although we have not been able to establish the direct link between PAT with circular detectors and the toroidal Radon transform, studying the toroidal Radon transform is an interesting geometric problem in its own right.

We assume that all tori are parallel to the $x_1x_2$-plane and consider two geometries: the centers of tori are located on a cylinder, or on a plane.
\begin{defi}
Let $u>0$ be a radius of the central circles of tori. 
Let $A\times\RR\subset\RR^2\times\RR$ be the set of the centers of tori.
The toroidal Radon transform $R_T$ maps $f\in C^\infty_c(\RR^3)$ into 
 \begin{equation}\label{eq:torusradon}
R_Tf(\boldsymbol\mu,p,r)=\displaystyle\frac{1}{2\pi }\intL^{2\pi}_0\int\limits^{2\pi}_0 f(\boldsymbol\mu+(u-r\cos\beta)\vec\alpha,p+r\sin\beta)d\beta d \alpha,
 \end{equation}
for $(\boldsymbol\mu,p,r)\in A\times\RR\times(0,\infty)$.
Here $\alpha$ is the angular parameter along the central circle, $(\boldsymbol\mu,p)$ is the center of the torus, and $\beta$ and $r$ are the polar angle and the radius of the tube of the torus, respectively.
\end{defi}

We consider two situations: $A$ is the circle $\partial B^2_R(0)$ or the line $x_1=0$. Thus the set of the centers of tori is a cylinder $\partial B^2_R(0)\times\RR$ or the $x_2x_3$-plane. 
We then present the relation between the circular Radon transform and the toroidal Radon transform.
This relation leads naturally to an inversion formula, if one uses an inversion formula for the circular Radon transform (already discussed in~\cite{finchhr07,kunyansky07} or \cite{andersson88,fawcett85,nattererw01,nilsson97,reddingn01}). 
\begin{defi}
Let $f$ be a compactly supported function in $\RR^3$.
The circular Radon transform $M_{}$ maps a function $f$ into 
$$
M_{}f(\boldsymbol\mu,x_3,r)=\intL^{2\pi}_0 f(\boldsymbol\mu+r\vec\alpha,x_3)d\alpha \quad\mbox{ for } (\boldsymbol\mu,x_3,r)\in A\times\RR\times(0,\infty).
$$
\end{defi}


\subsection{Reconstruction}\label{recontorus}
The inner integral with respect to $\beta$ in \eqref{eq:torusradon} can be thought of as the circular Radon transform with radius $r$.
As in subsection~\ref{recon}, we will first invert this transform.

Let us define the operator 
 $R^*_T$ for $g\in C^\infty_c(A\times\RR\times[0,\infty))$ by
$$
R^*_Tg(\boldsymbol\mu,z,\rho)=\displaystyle\intR g(\boldsymbol\mu,p,\sqrt{(z-p)^2+\rho^2})dp,
$$
where $(\boldsymbol\mu,z,\rho)\in A\times\RR^2$.
The following two lemmas show the relation between the circular and the toroidal Radon transforms.

Let us define the linear operator $I_2^{-1}$ by $\widehat{I^{-1}_2h}(\boldsymbol\mu,\boldsymbol\xi)=|\xi_2|\hat{h}(\boldsymbol\mu,\boldsymbol\xi)$,
where $h$ is a function on $A\times\RR^2$ and $\hat h$ is its Fourier transform in the last two-dimensional variable.
\begin{lem}\label{lem:andersson}
 Let $f\in C^\infty_c(\RR^3)$.
Then we have
\begin{equation}\label{eq:invtorus}
\displaystyle  \frac{1}{2} I^{-1}_2R^*_TR_Tf(\boldsymbol\mu,x_3, r)=\left\{\begin{array}{ll}M_{}f(\boldsymbol\mu,x_3,u-r)+M_{}f(\boldsymbol\mu,x_3,u+r)&\mbox{ if } u>r,\\\\
 M_{}f(\boldsymbol\mu,x_3,r-u)+M_{}f(\boldsymbol\mu,x_3,u+r)&\mbox{ otherwise, }\end{array}\right.
\end{equation}
\end{lem}

To prove this lemma, we follow the method discussed in~\cite{andersson88,nattererw01,nilsson97}.
\begin{proof}
By definition, we have
$$
R_Tf(\boldsymbol\mu,p,r)
=\displaystyle\frac{1}{2\pi }\sum^2_{j=1}\intL^{2\pi}_0\int\limits^1_{-1} f(\boldsymbol\mu+(u+(-1)^jr\sqrt{1-s^2})\vec\alpha,p+rs)\frac{ds}{\sqrt{1-s^2}}d\alpha.
$$
We take the Fourier transform of $R_Tf$ with respect to $p$ and the Hankel transform of order zero of $\widehat{R_Tf}$ with respect to $r$. 
Then $H_0\widehat{R_Tf}(\boldsymbol\mu,\xi_1,\eta)$ can be written as
\begin{equation}\label{hankelrtftorus}
\begin{array}{ll}
\displaystyle\frac{1}{2\pi }\sum^2_{j=1}\intL^{2\pi}_0\int\limits\half\int\limits\half\hat{f}(\boldsymbol\mu+(u+(-1)^jb)\vec\alpha, \xi_1)\cos(\rho\xi_1) J_0(\eta\sqrt{\rho^2+b^2})d\rho dbd\alpha,
  \end{array}
\end{equation}
where $\hat{f}$ and $\widehat{R_Tf}$ are the 1-dimensional Fourier transforms of $f$ and $R_Tf$ with respect to $z$ and $p$, respectively.
Lastly, we change variables $(r,s)\rightarrow (\rho,b)$, where $r=\sqrt{\rho^2+b^2}$ and $s=\rho/\sqrt{\rho^2+b^2}$.
Applying~\eqref{eq:batemann} to~\eqref{hankelrtftorus}, we get 
\begin{equation*}
H_0\widehat{R_Tf}(\boldsymbol\mu,\xi_1,|\boldsymbol\xi|)=\displaystyle\frac{1}{2\pi }\sum^2_{j=1}\intL^{2\pi}_0\int\limits\half\hat{f}(\boldsymbol\mu+(u+(-1)^jb)\vec\alpha, \xi_1)\dfrac{\cos(b\xi_2)}{\xi_2}dbd\alpha.
 \end{equation*}
The inner integral in the right hand side of the last equation is the Fourier cosine transform with respect to $b$, so taking the inverse Fourier cosine transform of the above formula, we get 
\begin{equation}\label{relationhankelandfouriertorus}
\displaystyle\sum^2_{j=1}\intL^{2\pi}_0 \hat{f}(\boldsymbol\mu+(u+(-1)^js)\vec\alpha, \xi_1)d\alpha= 4\int\limits\half H_0\widehat{R_Tf}(\boldsymbol\mu,\xi_1,|\boldsymbol\xi|)\cos(s\xi_2)\xi_2 d\xi_2.
\end{equation}
For a fixed $\xi_1$, one recognizes the sum of two circular Radon transforms on the left. 

Similarly to~\eqref{eq:relationhankelandback}, we can change the right hand side of~\eqref{relationhankelandfouriertorus} into a term containing operator $R_T ^*$, i.e.,
\begin{equation}\label{eq:relationhankelandbacktorus}
   \widehat{R^*_Tg}(\boldsymbol\mu, \boldsymbol\xi)=2\pi H_0\hat{g}(\boldsymbol\mu,\xi_1,|\boldsymbol\xi| ).
  \end{equation}
 Here $\widehat{R^*_Tg}$ is the 2-dimensional Fourier transform with respect to the variables $(z,\rho)$.
Combining~\eqref{eq:relationhankelandbacktorus} with~\eqref{relationhankelandfouriertorus}, we have for $g=R_Tf$,
\begin{equation*}\label{relationfourierandback}
\begin{array}{ll}
\displaystyle \sum^2_{j=1}\int\limits_{0}^{2\pi} \hat{f}(\boldsymbol\mu+(u+(-1)^js)\vec\alpha, \xi_1)d\alpha &\displaystyle=\frac{2}{\pi} \int\limits\half \widehat{R^*_Tg}(\boldsymbol\mu, \boldsymbol\xi) \cos(s\xi_2)\xi_2 d\xi_2\\
&\displaystyle=\frac{1}{\pi} \intR \widehat{R^*_Tg}(\boldsymbol\mu, \boldsymbol\xi) e^{is\xi_2}|\xi_2| d\xi_2,
\end{array}
\end{equation*}
since $\widehat{R^*_Tg}$ is even in $\xi_2$.
\end{proof}

\begin{lem}\label{lem:redding}
  Let $f\in C^\infty_c(\RR^3)$.
Then we have
$$
\begin{array}{l}
\displaystyle\frac{2}{\pi}\intR\intR \int\limits\half rsR_Tf(\boldsymbol\mu,-\eta,s)  e^{-i(s^2+2x_3\eta+x_3^2-\eta^2+r^2)\xi}\xi ds d\eta d\xi\\
=\left\{\begin{array}{ll}M_{}f(\boldsymbol\mu,x_3,u-r)+M_{}f(\boldsymbol\mu,x_3,u+r)&\mbox{ if } u>r,\\\\
 M_{}f(\boldsymbol\mu,x_3,r-u)+M_{}f(\boldsymbol\mu,x_3,u+r)&\mbox{ otherwise}.
 \end{array}\right.
 \end{array}
$$
To prove this lemma, we follow the method discussed in~\cite{reddingn01}.
\end{lem}

\begin{proof}
Let $G$ be defined by 
$$
G(\boldsymbol\mu,p,\xi):=\displaystyle\int\limits\half rR_Tf(\boldsymbol\mu,p,r) e^{-ir^2\xi}dr.
$$
Then we have
$$
\begin{array}{ll}
   G(\boldsymbol\mu,p,\xi)&=\displaystyle\frac{1}{2\pi}\int\limits\half\intL^{2\pi}_0\int\limits^\pi_{-\pi}rf(\boldsymbol\mu+(u-r\cos\beta)\vec\alpha,p+r\sin\beta)e^{-ir^2\xi}d\beta d\alpha dr\\
&=\displaystyle\frac{1}{2\pi}\intL^{2\pi}_0\intR\intR f(\boldsymbol\mu+(u-y)\vec\alpha,p+z) e^{-i(y^2+z^2)\xi}dydzd \alpha\\
&=\displaystyle\frac{1}{2\pi}\intL^{2\pi}_0\intR\intR f(\boldsymbol\mu+(u-y)\vec\alpha,z) e^{-i(y^2+(z-p)^2)\xi}dydzd \alpha\\
&=\displaystyle \frac{e^{-ip^2\xi}}{2\pi}\intL^{2\pi}_0\intR\intR f(\boldsymbol\mu+(u-y)\vec\alpha,z) e^{-i(y^2+z^2)\xi}e^{2ipz\xi}dydzd\alpha,
  \end{array}
  $$
where in the second line, we switched from the polar coordinates $(r,\beta)$ to the Cartesian coordinates $(y,z)\in\RR^2$.
Making the change of variables $r=y^2+z^2$ gives that $G(\boldsymbol\mu,p,\xi)$ equal 
$$
\begin{array}{l}
\displaystyle \frac{e^{-ip^2\xi}}{2\pi}\sum^2_{j=1}\intL^{2\pi}_0\intR\intR f(\boldsymbol\mu+(u+(-1)^j\sqrt{r-z^2})\vec\alpha,z) \frac{e^{-ir\xi}e^{2ipz\xi}}{2\sqrt{r-z^2}}drdzd\alpha.
  \end{array}
  $$
Let us define the function 
$$
k_{\boldsymbol\mu}(\alpha,z,r):=\left\{\begin{array}{ll}\displaystyle\sum^2_{j=1}f(\boldsymbol\mu+(u+(-1)^j\sqrt{r-z^2})\vec\alpha,z)/\sqrt{r-z^2} \qquad &\mbox{if } 0<z^2<r,\\
                        0\qquad &\mbox{otherwise.}
                       \end{array}\right.
                       $$
Then we have 
$$
\begin{array}{ll}
   G(\boldsymbol\mu,p,\xi)&=\displaystyle\frac{e^{-ip^2\xi}}{4\pi} \intL^{2\pi}_0\intR\intR  k_{\boldsymbol\mu}(\alpha,z,r) e^{-ir\xi}e^{2ipz\xi}drdzd \alpha\\
&=\displaystyle\frac{e^{-ip^2\xi}}{4\pi} \intL^{2\pi}_0\widehat{k_{\boldsymbol\mu}}(\alpha,-2p\xi,\xi)d \alpha,
   \end{array}
  $$
  where $\widehat{k_{\boldsymbol\mu}}$ is the 2-dimensional Fourier transform of $k_{\boldsymbol\mu}$ with respect to the variables $(z,r)$.
Also, we have
\begin{equation*}\label{reddingandradon}
\begin{array}{ll}
\displaystyle\sum^2_{j=1}\intL^{2\pi}_0f(\boldsymbol\mu+(u+(-1)^js)\vec\alpha,x_3)d\alpha=\displaystyle \intL^{2\pi}_0 sk_{\boldsymbol\mu}(\alpha,x_3,x_3^2+s^2)d\alpha\\
=\displaystyle\frac{1}{4\pi^2}\intRR \intL^{2\pi}_0 s\widehat{k_{\boldsymbol\mu}}(\alpha,\eta,\xi)e^{-i(x_3\eta+(x_3^2+s^2)\xi)}d\alpha d\eta d\xi\\
=\displaystyle\frac{1}{\pi}\intR\intR se^{i\frac{\eta^2}{4\xi}}G(\boldsymbol\mu,-\frac{\eta}{2\xi},\xi)e^{-i(x_3\eta+(x_3^2+s^2)\xi)}d\eta d\xi\\
=\displaystyle\frac{2}{\pi}\intR\intR sG(\boldsymbol\mu,-\eta,\xi)e^{-i(2x_3\eta+(x_3^2+s^2)-\eta^2)\xi}\xi d\eta d\xi,
  \end{array}
  \end{equation*}
  where in the last line, we changed the variable $\eta/2\xi$ to $\eta$.
\end{proof}
\subsubsection{Cylindrical geometry}
Let the centers of the central circles be located on the a cylinder $\partial B^2_R(0)\times\RR=A\times\RR$.
That is, $A$ is the circle centered at the origin with radius $R$.
The next two results show that the circular Radon transform can be recovered  from the toroidal Radon transform. 
Both theorems are easily obtained using Lemma~\ref{lem:andersson}.
\begin{thm}\label{thm:inversionwithcondition}
If $R/2<u<R$ and $f\in C^\infty_c(B^2_{R-u}(0)\times\RR)$, then 
\begin{equation*}
\displaystyle M_{} f(\boldsymbol\mu,x_3,r)=\left\{\begin{array}{ll}2^{-1}I^{-1}_2R^*_TR_Tf(\boldsymbol\mu,x_3, r-u) \quad&\mbox{ if } r>u,\\
                                       0&\mbox{otherwise.}
                                      \end{array}\right. 
\end{equation*}
 
\end{thm}

\begin{thm}\label{thm:inversion}
Let $f\in C^\infty_c(B^2_R(0)\times\RR)$.
Then 
\begin{equation*}
\displaystyle M_{}f(\boldsymbol\mu,x_3,r)=\left\{\begin{array}{ll}\displaystyle\frac{1}{2}\sum^{[\frac Ru+\frac12]}_{j=0}(-1)^jI^{-1}_2R^*_TR_Tf(\boldsymbol\mu,x_3, (2j+1)u-r)&\mbox{ if }r\leq u,\\
\displaystyle\frac{1}{2}\sum^{[R/u]}_{j=0}(-1)^{j}I^{-1}_2R^*_TR_Tf(\boldsymbol\mu,x_3, (2j+1)u+r)&\mbox{ otherwise.}
\end{array}\right.
\end{equation*}
\end{thm}
\begin{rmk}
One can obtain other relations similar to Theorems~\ref{thm:inversionwithcondition} and~\ref{thm:inversion}, by using Lemma~\ref{lem:redding} instead of Lemma~\ref{lem:andersson}.
\end{rmk}
\begin{rmk}
When the set of  {centers of the circular detectors} is a cylinder, (i.e., $A$ is a circle,) one can recover $f$ from its torodial transform $R_Tf$ by applying inversion formulas for the spherical Radon transform with the centers of circles located on the circle (see e.g.~\cite{finchhr07,kunyansky07,kunyansky071}) to the left hand sides of equations in Theorems~\ref{thm:inversionwithcondition} and \ref{thm:inversion}.
\end{rmk}

\begin{rmk}
If $u>2R$ (i.e., the radius of central circles is bigger than the diameter of the cylinder $B^2_R(0)\times\RR$), then 
$$
M_{}f(\boldsymbol\mu,x_3,r)=\left\{\begin{array}{ll}2^{-1}I^{-1}_2R^*_TR_Tf(\boldsymbol\mu,x_3,u-r)&\mbox{ if } r\leq u,\\\\
2^{-1}I^{-1}_2R^*_TR_Tf(\boldsymbol\mu,x_3,u+r)&\mbox{ otherwise.}
\end{array}\right.
$$
\end{rmk}

\subsubsection{Planar geometry}
Let $A\subset\RR^2$ be the $x_1=0$ line (i.e., the centers of tori are located on the $x_2x_3$-plane in $\RR^3$).
Then $R_Tf(\boldsymbol\mu,x_3,r)$ is equal to zero if $f$ is an odd function in $x_1$.
We thus assume the function $f$ to be even in $x_1$.

\begin{thm}\label{thm:inversioninline}
Let $f\in C^\infty_c(B^3_R(0))$ be even in $x_1$. 
Then we have
\begin{equation}\label{eq:toriii}
\displaystyle M_{}f(\boldsymbol\mu,x_3,r)=\left\{\begin{array}{ll}\displaystyle\frac{1}{2}\sum^{[\frac {R+u}{2u}]}_{j=0}(-1)^jI^{-1}_2R^*_TR_Tf(\boldsymbol\mu,x_3, (2j+1)u-r)&\mbox{ if }r\leq u,\\
\displaystyle\frac{1}{2}\sum^{[R/2u]}_{j=0}(-1)^{j}I^{-1}_2R^*_TR_Tf(\boldsymbol\mu,x_3, (2j+1)u+r)&\mbox{ otherwise.}
\end{array}\right.
\end{equation}
\end{thm}
\begin{rmk}
When $A$ is a line, we can determine $f$ from $R_Tf$ by applying inversion formulas for the spherical Radon transform with the centers of circles on the hyperplane~\cite{andersson88,fawcett85,nattererw01,nilsson97,reddingn01} to the left hand side of~\eqref{eq:toriii}.
\end{rmk}
\begin{rmk}
If $u>R$ (i.e., the radius of the detectors is bigger than the radius of the ball containing $supp\; f$), then 
$$
M_{}f(\boldsymbol\mu,x_3,r)=\left\{\begin{array}{ll}2^{-1}I^{-1}_2R^*_TR_Tf(\boldsymbol\mu,x_3,u-r)&\mbox{ if } r<u,\\\\
2^{-1}I^{-1}_2R^*_TR_Tf(\boldsymbol\mu,x_3,u+r)&\mbox{ otherwise.}
\end{array}\right.$$
\end{rmk}
\section{Conclusion}
Here we studied a Radon-type transform arising in PAT with circular detectors, and also the toroidal Radon transform. 
We proved that these transforms reduce to well-studied transforms: the Radon transform over circles with a fixed radius, or the circular Radon transform.

\bibliographystyle{9}


\end{document}